\documentclass{article}

\usepackage[hidelinks]{hyperref}
\usepackage{color}
\usepackage{fullpage}
\usepackage{float}

\usepackage{graphics,amsmath,amssymb}
\usepackage{amsthm}
\usepackage{amsfonts}
\usepackage{latexsym}
\usepackage{epsf}
\newtheorem{theorem}{Theorem}[section]
\newtheorem{lemma}[theorem]{Lemma}
\newtheorem{corollary}[theorem]{Corollary}

\newtheorem{definition}[theorem]{Definition}

\numberwithin{equation}{section}

\begin{document}
	
	\begin{center}
		\vskip 1cm{\LARGE\bf On Schizophrenic Patterns in $b$-ary Expansions of Some Irrational Numbers}
		\vskip 1cm
		\large
		L\'aszl\'o T\'oth\\
		Rue des Tanneurs 7 \\
		L-6790 Grevenmacher \\
		Grand Duchy of Luxembourg \\
		\href{mailto:uk.laszlo.toth@gmail.com}{\tt uk.laszlo.toth@gmail.com}
	\end{center}
	
	\vskip .2 in

	\begin{abstract}
		In this paper we study the $b$-ary expansions of the square roots of the function defined by the recurrence $f_b(n)=b f_b(n-1)+n$ with initial value $f(0)=0$ taken at odd positive integers $n$, of which the special case $b=10$ is often referred to as the ''schizophrenic'' or ''mock-rational'' numbers. Defined by Darling in $2004$ and studied in more detail by Brown in $2009$, these irrational numbers have the peculiarity of containing long strings of repeating digits within their decimal expansion. The main contribution of this paper is the extension of schizophrenic numbers to all integer bases $b\geq2$ by formally defining the schizophrenic pattern present in the $b$-ary expansion of these numbers and the study of the lengths of the non-repeating and repeating digit sequences that appear within.
	\end{abstract}

	\section{Introduction} \label{intro}
	Schizophrenic numbers are often defined as a subclass of irrational numbers referred to as \emph{zebra numbers}, i.e., irrational numbers whose decimal expansion appears to be rational for periods. Delahaye \cite[pp.\ 134--142]{Delahaye04} enumerates several kinds of zebra numbers, such as $\sqrt{10^{60}-1}$ (a special case of Y\'el\'ehada numbers of the form $(10^n-1)^{1/k}$) whose decimal expansion contains a certain pseudo-periodicity, Robert Israel numbers which are of the form $k(n) = \sqrt{\frac{9}{121} 100^n + \frac{112-44n}{121}}$, which present striking ordered patterns in the decimal expansion for $n=95$, and schizophrenic (or mock-rational) numbers.
	
	Darling \cite[p.\ 12]{Darling04} defines schizophrenic numbers as square roots at odd $n$ of the function defined by the recurrence $f(n)=10f(n-1)+n$ with initial value $f(0)=0$ (sequence {A014824} in the OEIS). He observes that, unexpectedly, the decimal expansions of the resulting irrational numbers contain strings of repeating digits. He gives the example of $\sqrt{f(49)}$:
	\begin{center}
		\begin{tabular}{rllll}
			$1.1111111111$& $1111111111$& $1111111111$& $1111111111$& $1111110860$ \\
			$5555555555$  & $5555555555$& $5555555555$& $5555555555$& $5555527305$ \\ 
			$4166666666$  & $6666666666$& $6666666666$& $6666666666$& $6660296260$ \\
			$3472222222$  & $2222222222$& $2222222222$& $2222222222$& $\ldots$ \\ 
			$\times10^{24}$.
		\end{tabular}
	\end{center}
	
	The lengths of the repeating digits progressively become shorter and after a while disappear completely. Brown \cite{Brown09} studies the repeating digits in the decimal expansion of $\sqrt{f(n)}$ for odd positive integers $n=2k-1$ and shows that these are closely related to the Taylor expansion of the square root of the recurrence's solution. Moreover, the author gives a short algorithm using this expansion to predict the repeating digit sequences. 
	
	In this paper we extend the above description of schizophrenic numbers to integer bases $b\geq2$ by formally defining the \textit{schizophrenic pattern} found in the $b$-ary expansion of some irrational numbers. In particular, we study the recurrence defined for positive integers $n$ and $b\geq2$ by $f_b(n)=bf_b(n-1)+n$ and initial value $f_b(0)=0$ and show that the Taylor expansion of the square root of its solution, taken at odd positive integers, allows to construct such schizophrenic patterns in positive integer bases $b\geq2$.  An example is the base-$8$ expansion of $\sqrt{f_8(49)}$:
	\begin{center}
		\begin{tabular}{rllll}
			$1.1111111111$& $1111111111$& $1111111111$& $1111111111$& $1111110600$ \\
			$4444444444$  & $4444444444$& $4444444444$& $4444444444$& $4444402144$ \\ 
			$3333333333$  & $3333333333$& $3333333333$& $3333333333$& $3317512442$ \\ 
			$2666666666$  & $6666666666$& $6666666666$& $6666666666$& $\ldots_8$ \\ 
			$\times8^{24}$.
		\end{tabular}
	\end{center}
	We examine the lengths of the repeating and non-repeating digit blocks that appear within the $b$-ary expansion of such numbers, which allows us to conclude that these patterns grow in size with $n$. This in turn gives rise to the definition of  a \textit{schizophrenic sequence}, a monotonically increasing sequence of irrational numbers $S=(\sqrt{f_b(1)}, \sqrt{f_b(3)}, \ldots, \sqrt{f_b(2k-1)}, \ldots)$ whose terms contain a schizophrenic pattern that grows in size with $n$. We then examine several properties of $b^m$-ary schizophrenic patterns for positive nonzero integer $m$ and examine the relationships between schizophrenic patterns in such bases. 
	
	Throughout the paper we complement our results with illustrative numerical examples computed using Wolfram Mathematica, and we provide the corresponding Mathematica code in the final section. Note that in the remainder of this paper we shall denote by $n_b$ the base-$b$ representation of the number $n$ and all digits greater than $9$ shall be denoted by their commonly accepted alphabetic counterparts; thus, $10$ is denoted by $a$, $11$ by $b$, and so forth.
	
	Finally, we would like to mention that the recurrences $f_b(n)$ are present in the OEIS as sequences {A000340} (for $b=3$, $n<25$), {A014825} ($b=4$, $n<24$), {A014827} ($b=5$, $n<22$), {A014829} ($b=6$, $n<21$) and {A014824} (for $b=10$, $n<21$).  
	
	\section{The schizophrenic pattern} \label{section-schizo}
	
	Let $n>0$ and $b\geq2$ be positive integers and define the recurrence 
	$$f_b(n) = b f_b(n-1) + n$$
	with initial value $f_b(0) = 0$. We begin this section by examining the solution of this recurrence and how its square root taken at positive odd integers generates blocks of repeating and non-repeating digits in its $b$-ary expansion. The solution of the recurrence is easily determined:
	$$
	f_b(n) = \frac{b^{n+1}-b(n+1)+n}{(b-1)^2}.
	$$
	Now let $n=2k-1$ for positive integer $k>0$. We have
	$$
	f_b(2k-1) = \left(\frac{b^{k}}{b-1}\right)^2 \left(1-\frac{(2k-1)(b-1)+b}{b^{2k}}\right).
	$$  
	Taking the square root and using the Taylor expansion for the rightmost factor, we obtain
	\begin{align*}
	\sqrt{f_b(2k-1)} &= \frac{b^k}{b-1} \sum_{l=0}^{\infty}(-1)^l \binom{1/2}{l} \left(\frac{(2k-1)(b-1)+b}{b^{2k}}\right)^l \\
	&= \frac{b^k}{b-1} \left(1 - \left(\frac{1}{2}\right)\left(\frac{(2k-1)(b-1)+b}{b^{2k}}\right) - \left(\frac{1}{8}\right)\left(\frac{((2k-1)(b-1)+b)^2}{b^{4k}}\right) -  \ldots\right).
	\end{align*}
	Our aim is now to show that each term in the Taylor expansion above generates a $b$-ary digit block beginning with a non-repeating digit sequence followed by a repeating digit sequence. Once these blocks are added together, a pattern emerges in the $b$-ary expansion of $\sqrt{f_b(2k-1)}$. Let us now denote by $\tau_l$ the $l^{\rm th}$ term in the Taylor expansion above, i.e.,
	$$
	\tau_l = \frac{b^k}{b-1} (-1)^l \binom{1/2}{l} \left(\frac{(2k-1)(b-1)+b}{b^{2k}}\right)^l.
	$$
	We now examine the various digit contributions that make up the $b$-ary expansion of $\tau_l$. To do this, let us denote $\tau_{l_1}=(-1)^l \binom{1/2}{l}$ the $l^{\rm th}$ binomial coefficient in the Taylor expansion and let $\tau_{l_2}=\frac{b^k}{b^{2kl}}$ and $\tau_{l_3}=\frac{((2k-1)(b-1)+b)^l}{b-1}$, so  $\tau_l=\tau_{l_1}\tau_{l_2}\tau_{l_3}$. First, it is clear that the $b$-ary expansion of of $\tau_l$ will contain an infinitely repeating digit sequence, which is due to the denominator in $\tau_{l_3}$. The contribution of $\tau_{l_2}$ shifts the $b$-ary digits to the right of the radix point. The non-repeating $b$-ary digit sequence before and after the radix point in $\tau_l$ is thus contributed by $\tau_{l_1}\tau_{l_3}$. The following Lemma establishes the length of this non-repeating digit sequence.
	
	\begin{lemma} \label{taylor-non-repeating}
		Let $q$ denote the exponent of $2$ in the denominator of $\tau_{l_1}$ and let $r$ be the smallest positive integer such that $2^q$ divides $b^r$, or $0$ if $b$ is odd. Then the non-repeating digit sequence in the $b$-ary expansion of $\tau_{l_1}\tau_{l_3}$ has length 
		$$
		\lfloor \log_b ( | \tau_{l_1}\tau_{l_3} | ) \rfloor + 1 + r.
		$$
	\end{lemma}
	
	\begin{proof}
		The $b$-ary expansion of $\tau_{l_1}\tau_{l_3}$ will have a non-repeating digit block before the radix point, a non-repeating digit block after the radix point, and a repeating digit block. It is clear that $\lfloor \log_b ( | \tau_{l_1}\tau_{l_3} | )\rfloor +1$ is the length of the non-repeating digit block before the radix point. The length of the non-repeating part after the radix point is determined by elementary means knowing that the denominator of $\tau_{l_1}\tau_{l_3}$ will be of the form $(b-1)2^q$, for some positive integer $q$, due to the binomial coefficient. Indeed, if $2^q$ divides a power of $b$, this length is equal to the smallest positive integer $r$ such that $2^q|b^r$. Otherwise (i.e., for all odd $b$), we set $r=0$. The combined length of the non-repeating digit blocks before and after the radix point is thus equal to $\lfloor \log_b ( | \tau_{l_1}\tau_{l_3} | ) \rfloor + 1 + r$. 
	\end{proof}
	
	Note that in the special case of Brown's ''mock-rational number'', i.e., when $b=10$, we have $n=q$. 
	
	We are now ready to reconstruct the $b$-ary expansion of $\sqrt{f_b(2k-1)}$ using the terms in its Taylor expansion.
	
	\subsection{Schizophrenic blocks} \label{sec-blocks}
	
	The purpose of this section is to use the Taylor expansion described above to define the digit blocks that build up the $b$-ary expansion of $\sqrt{f_b(2k-1)}$. As we will show, each such ''building block'' consists of a non-repeating digit sub-block followed by a repeating digit sub-block and the concatenation in base $b$ of these building blocks constitutes the $b$-ary expansion of $\sqrt{f_b(2k-1)}$. In this section, we will examine the lengths of these sub-blocks before looking at their contribution to certain properties of schizophrenic patterns. To begin, we have the following definition.
	
	\begin{definition}[Schizophrenic block] \label{def-schiz-block}
		Let $n>0$, $k>1$ and $b\geq2$ be positive integers and define the recurrence $f_b(n) = b f_b(n-1) + n$
		with initial value $f_b(0) = 0$. We then define a \textit{schizophrenic block} a block of $b$-ary digits within the $b$-ary expansion of $\sqrt{f_b(2k-1)}$ beginning with a non-repeating digit block followed by a repeating digit block, related to a single term in its Taylor expansion.
	\end{definition}
	
	As we have shown earlier, the $l^{\rm th}$ term in the Taylor expansion, i.e., $\tau_l=\tau_{l_1}\tau_{l_2}\tau_{l_3}$ contributes a non-repeating digit sub-block followed by a repeating digit sub-block to the $b$-ary expansion of $\sqrt{f_b(2k-1)}$. As we will show below, the corresponding schizophrenic block will be very similar to $\tau_l$, with the sole difference of the  leading non-repeating digit sub-block whose beginning is altered due to the repeating digit dub-block of the previous ($(l-1)^{\rm th}$) term in the Taylor expansion. In this section we will examine this in detail. Note that in the following, we will refer to the schizophrenic block generated by the $l^{\rm th}$ term in the Taylor expansion of $\sqrt{f_b(2k-1)}$ as the ''$l^{\rm th}$ schizophrenic block'' and denote it by $s_l$. For instance, the first order term in the Taylor expansion related to $\sqrt{f_{10}(49)}$ generates the following schizophrenic block in its decimal expansion: 
	$$
	0860555555555555555555555555555555555555555555555,
	$$
	it is therefore its $1^{\rm st}$ schizophrenic block, starting with a non-repeating sub-block of length $4$ and a repeating digit sub-block of length $45$. The $2^{\rm nd}$ schizophrenic block is 
	$$
	273054166666666666666666666666666666666666666666,
	$$
	with a non-repeating sub-block of length $7$ and a repeating sub-block of length $41$, and so forth. Note that the digit block before the $1^{\rm st}$ schizophrenic block is always composed solely of $1$s (in base $b$), which is due to the denominator in $\tau_{l_3}$.	
	
	We will now examine the lengths of these sub-blocks and the length of the $l^{\rm th}$ schizophrenic block in more detail. In order to calculate these lengths, we look at the digit contributions of the $l^{\rm th}$ term in the Taylor expansion as well as the contributions of the surrounding schizophrenic blocks. We begin with the length of the non-repeating digit sub-block in the following Theorem.
	
	\begin{theorem} [Non-repeating digit block length] \label{theo-length-nonrepeating}
		Assume the same notation as in Lemma \ref{taylor-non-repeating}, and define the function
		$$
		\epsilon(l) =
		\begin{cases}
		1  & \text{if $d_{l-1}-f_l < 0$,} \\
		0 & \text{otherwise,}
		\end{cases}
		$$
		where $f_l$ is the first digit of $\tau_{l_1}\tau_{l_3}$ and $d_l$ is the repeating digit within the $l^{\rm th}$ schizophrenic block. Then the $l^{\rm th}$ schizophrenic block within the schizophrenic pattern of $\sqrt{f_b(2k-1)}$ begins with a non-repeating digit sub-block of length 
		$$
		\lfloor \log_b ( | \tau_{l_1}\tau_{l_3} | ) \rfloor + 1 + r + \epsilon(l).
		$$
	\end{theorem}
	
	\begin{proof}
		Lemma \ref{taylor-non-repeating} establishes that the leading non-repeating digit sequence in the $b$-ary expansion of the $l^{\rm th}$ term in the Taylor expansion of $\sqrt{f_b(2k-1)}$ has length $\lfloor \log_b ( | \tau_{l_1}\tau_{l_3} | ) \rfloor + 1 + r$. When we re-constructing the $l^{\rm th}$ schizophrenic block by subtracting the $l^{\rm th}$ term in the Taylor expansion from the sum of the previous terms, we need to account for a possible additional digit appearing in the non-repeating sub-block. We compensate for this additional digit with the function $\epsilon(l)$. Note that $d_0$ is always equal to $1$, as the digit block situated before the $1^{\rm st}$ schizophrenic block is always composed of a string of $1$s in base $b$.
	\end{proof}
	
	Using the example of $\sqrt{f_{10}(49)}$ above, the length of the non-repeating digit block in the decimal expansion of its $1^{\rm st}$ schizophrenic block is $4$ since $\lfloor \log_{10} (\frac{49 \times 9 + 10}{9\times2} ) \rfloor + 3=4$ (note that $d_{0}=1$, $f_1=2$ so $\epsilon(1)=1$ and $r=1$ as it is the smallest positive integer such that $2|10^r$). The rest of the schizophrenic block is composed of a repeating $b$-ary digit sequence. Before calculating its length, we examine that of the entire $l^{\rm th}$ schizophrenic block in the following Theorem.
	
	\begin{theorem}[Schizophrenic block length] \label{theo-length-block}
		Assume the same notation as in Theorem \ref{theo-length-nonrepeating} and denote by $\lambda_l$ the length of the $l^{\rm th}$ schizophrenic block. Then
		$$
		\lambda_l = 2k(l+1) - (\lfloor \log_b ( | \tau_{({l+1})_1}\tau_{({l+1})_3} | ) \rfloor + 1 + \epsilon(l+1)) - \sum_{i=0}^{l-1}\lambda_i,
		$$
		where $\lfloor m\rfloor$ denotes the greatest integer less than or equal to $m$.
	\end{theorem}
	
	\begin{proof}
		In order to calculate the length of the $l^{\rm th}$ schizophrenic block, we need to consider the sum of the lengths of the previous schizophrenic blocks and the position where the $(l+1)^{\rm th}$ schizophrenic block begins. The desired length will then be the difference between the latter and the former. 
		
		Consider therefore the $(l+1)^{\rm th}$ term in the Taylor expansion. Clearly, the amount of zeros after the radix point in its $b$-ary expansion will be $2k(l+1)$ minus the amount of digits before the radix point in the digit contribution of this term. Notice that this contribution is equal to the absolute value of $\tau_{({l+1})_1}\tau_{({l+1})_3}$. It is now clear that the amount of zeros after the radix point in the $b$-ary expansion of the  $(l+1)^{\rm th}$ term of the Taylor expansion is
		$$
		2k(l+1) - (\lfloor \log_b ( |\tau_{({l+1})_1}\tau_{({l+1})_3}| ) \rfloor + 1).
		$$
		As with the case of Theorem \ref{theo-length-nonrepeating}, when we re-construct the $b$-ary expansion of $\sqrt{f_b(2k-1)}$ , we need to account for a possible additional digit appearing in the non-repeating sub-block of the $(l+1)^{\rm th}$ schizophrenic block. In order to compensate for this additional digit, we call upon the function $\epsilon$ once more. Thus, denoting the length of the $l^{\rm th}$ schizophrenic block within the schizophrenic pattern of $\sqrt{f_b(2k-1)}$ by $\lambda_l$, we conclude that
		$$
		\lambda_l = 2k(l+1) - (\lfloor \log_b ( |\tau_{({l+1})_1}\tau_{({l+1})_3}| ) \rfloor + 1 + \epsilon(l+1)) - \sum_{i=0}^{l-1}\lambda_i.
		$$
	\end{proof}
	
	Thus in the case of Brown's ''mock-rational number'' $\sqrt{f_{10}(49)}$ as described in Section \ref{intro}, by defining $\Lambda=\{\lambda_0,\lambda_1,\ldots,\lambda_n,\ldots\}$ its sequence of schizophrenic block lengths, we have $\Lambda=\{47,49,48,\ldots\}$. To illustrate this, we take the example of $\lambda_1$. We have $\tau_{2_1}\tau_{2_3} = \left|\binom{1/2}{2} \frac{(49*9+10)^2}{9} \right| = 2825.01388\ldots$ and $\epsilon(2)=0$, therefore
	\begin{align*}
	\lambda_1 &= 50 \times 2 - \left(\lfloor \log_{10} 2825.0138 \ldots ) \rfloor + 1\right) - 47 \\
	&= 100 - (\lfloor 3.4510\ldots \rfloor +1) - 47 \\
	&= 49.
	\end{align*}
	Note that here we have denoted by $\lambda_0$ the length of the repeating digit block of $1$s before the $1^{\rm st}$ schizophrenic block. As shown earlier, this block does not contain a non-repeating digit sequence within its $b$-ary expansion. Note also that the lengths of these schizophrenic blocks is at most $2k$.
	
	We now turn our attention to the length of the repeating digit sub-blocks. The following Theorem follows trivially from Theorems \ref{theo-length-nonrepeating} and \ref{theo-length-block} and is thus given without proof.
	
	\begin{theorem} [Repeating digit block length] \label{theo-length-repeating}
		Assume the same notation as in Theorems \ref{theo-length-block} and \ref{theo-length-nonrepeating}. Then the repeating digit sub-block within the $l^{\rm th}$ schizophrenic block has length 
		$$
		2k(l+1) - (\lfloor \log_b ( |\tau_{({l+1})_1}\tau_{({l+1})_3}| ) \rfloor + \lfloor \log_b ( |\tau_{{l}_1}\tau_{{l}_3}| ) \rfloor) - (\epsilon(l+1) + \epsilon(l)) - (r+2) - \sum_{i=0}^{l-1}\lambda_i.
		$$
	\end{theorem}
	
	We  are now ready to define the \textit{schizophrenic pattern} found in the $b$-ary expansion of some irrational numbers.
	
	\begin{definition}[Schizophrenic pattern] \label{def-schizo} Let $n>0$, $k>1$ and $b\geq2$ be positive integers and define the recurrence $f_b(n) = b f_b(n-1) + n$ with initial value $f_b(0) = 0$. Moreover, let $S=\{ \sigma_0, \sigma_1, \sigma_2, \ldots \}$ denote the set of consecutive schizophrenic blocks in the $b$-ary expansion of $\sqrt{f_b(2k-1)}$. Then $S$ constitutes its \textit{schizophrenic pattern}.
	\end{definition}
	
	In other words, the successive non-repeating and repeating $b$-ary digit blocks found in the expansion of $\sqrt{f_b(2k-1)}$, whose lengths are given by previous theorems in terms of schizophrenic blocks, form its schizophrenic pattern. We illustrate this definition with the following two examples, one for $b$ even and one for odd. We first examine the case $b=8$, in particular $\sqrt{f_8(49)}$ which in base $10$ yields the following number:
	\begin{center}
		\begin{tabular}{rllll}
			$5.3969902661$& $3673738708$& $1142857142$& $8571428571$& $4219350766$ \\
			$8557456653$  & $9888812550$& $4786980205$& $3485557081$& $3531722463$ \\ 
			$3909432718$  & $6075184755$& $7422859460$& $2682654865$& $0569395455$ \\
			$7886713749$  & $9323823666$& $8040585242$& $4676748726$& $\ldots$ \\ 
			$\times10^{21}$.
		\end{tabular}
	\end{center}
	
	In base $10$ therefore no schizophrenic pattern can be observed. However, in base $8$ the same number begins as follows:
	\begin{center}
		\begin{tabular}{rllll}
			$1.1111111111$& $1111111111$& $1111111111$& $1111111111$& $1111110600$ \\
			$4444444444$  & $4444444444$& $4444444444$& $4444444444$& $4444402144$ \\ 
			$3333333333$  & $3333333333$& $3333333333$& $3333333333$& $3317512442$ \\ 
			$2666666666$  & $6666666666$& $6666666666$& $6666666666$& $\ldots_8$ \\ 
			$\times8^{24}$,
		\end{tabular}
	\end{center}
	and the pattern emerges. More interesting digit sequences appear for odd $b$. Indeed, instead of repeating digits we observe repeating \emph{digit patterns}. For instance, when $b=11$, we have, for $\sqrt{f_{11}(49)}$ in base $11$:
	\begin{center}
		\begin{tabular}{rllll}
			$1.1111111111$& $1111111111$& $1111111111$& $1111111111$& $1111110990$ \\
			$6060606060$  & $6060606060$& $6060606060$& $6060606060$& $6060592135$ \\ 
			$a045a045a0$  & $45a045a045$& $a045a045a0$& $45a045a045$& $a04103a121$ \\
			$79a7245179$  & $a7245179a7$& $245179a724$& $5179a72451$& $\ldots_{11}$ \\ 
			$\times11^{24}$.
		\end{tabular}
	\end{center}
	The pattern then progressively disappears. We shall now examine several properties of schizophrenic patterns that follow from the results in this section.
	
	\subsection{Properties}
	
	We begin this section with a Corollary derived from the results in the previous section.
	
	\begin{corollary} \label{coro-schizo}
		Let $\lambda_{\rm max}$ denote the length of the longest repeating digit sub-block within a schizophrenic pattern in positive integer base $b\geq2$ and let $m$ be the greatest positive integer such that $\lfloor\frac{\lambda_{\rm max}}{m}\rfloor=2$. Then the schizophrenic pattern is also schizophrenic in bases $b^{m_i}$ for all positive integers $m_i$ such that $1<m_i\leq m$.
	\end{corollary}
	
	\begin{proof}
		The conversion of $\sqrt{f_b(2k-1)}$ from the integer base $b$ to another integer base $b^{m_i}$ involves regrouping its digits into groups of size $m_i$, then replacing each group with the digit that corresponds to the resulting number in base $b^{m_i}$. The repeating digit blocks will therefore remain repeating as long as $1<m_i\leq m$ where $m$ is a positive integer that satisfies $\lfloor\frac{\lambda_{\rm max}}{m}\rfloor=2$. 
	\end{proof}
	
	Definition \ref{def-schizo} together with the proof of Corollary \ref{coro-schizo} give rise to the following Lemma.
	
	\begin{lemma} \label{lemma-schizo}
		Let $\beta=b^j$ for positive integers $j$ and $b\geq2$. Then, for positive integers $n$, $k>0$ and a recurrence $f_\beta(n)$ defined by $f_\beta(n) = \beta f(n-1) + n$ with initial value $f_\beta(0) = 0$, the numbers $\sqrt{f_\beta(2k-1)}$ contain a schizophrenic pattern in all bases $b^{m_i}$, for  positive integers $m_i$ such that $0<m_i\leq m$, where $m$ is the greatest positive integer satisfying $\lfloor\frac{\lambda_{\rm max}}{m}\rfloor=2$, and $\lambda_{\rm max}$ denotes the length of the longest repeating digit sub-block within the $\beta$-ary schizophrenic pattern.
	\end{lemma}
	
	\begin{proof}
		By Corollary \ref{coro-schizo}, the Lemma holds for positive integers $m_i$ between $j$ and $m$. Since $\sqrt{f_\beta(2k-1)}$ contains a schizophrenic pattern in base $\beta=b^j$, the repeating digits in its $\beta$-ary representation are expanded into digit groups of larger size in bases $b^{m_i}$ for $m_i<j$. Since these digits are repeating, so are those in bases $b^{m_i}$, $m_i<j$.
	\end{proof}
	
	We illustrate this with the example of $\sqrt{f_3(49)}$, letting $m_1=2$ and $m_2=3$. This number begins, in base $3$, as:
	\begin{center}
		\begin{tabular}{rllll}
			$1.1111111111$& $1111111111$& $1111111111$& $1111111111$& $1111111111$ \\
			$1111101200$  & $2020202020$& $2020202020$& $2020202020$& $2020202020$ \\ 
			$2011010102$  & $0012001200$& $0012001200$& $1200120012$& $0012001200$ \\
			$1021120020$  & $2112100021$& $1210002112$& $1000211210$& $\ldots_{3}$ \\ 
			$\times3^{24}$.
		\end{tabular}
	\end{center}
	The schizophrenic pattern is apparent. Now examine the representation of the same number in base $3^{m_1}=9$:
	\begin{center}
		\begin{tabular}{rllll}
			$1.4444444444$& $4444444444$& $4435066666$& $6666666666$& $6666664112$ \\
			$0505050505$  & $0505050503$& $3750675307$& $5307530753$& $0740552382$ \\ 
			$4225078164$  & $4731127658$& $2207712484$& $0766815054$& $\ldots_{9}$ \\
			$\times9^{12}$.
		\end{tabular}
	\end{center}
	In base $9$ the number still contains a schizophrenic pattern. Finally, in base $3^{m_2}=27$:
	\begin{center}
		\begin{tabular}{rllll}
			$1.dddddddddd$& $ddddd526k6$& $k6k6k6k6k6$& $ja3if51if5$& $1if51ia7f6$ \\
			$ml2e97g0ml$  & $2d1n787b7i$& $njjdm13pjq$& $6ina7hc7k8$& $\ldots_{27}$ \\ 
			$\times27^{8}$.
		\end{tabular}
	\end{center}
	As we can see the schizophrenic pattern in these bases is preserved. We now turn our attention to sequences of schizophrenic numbers.

	\subsection{Schizophrenic sequences} \label{section-schizo-sequence}
	
	Let $f_b$ denote once more the recurrence defined in Definition \ref{def-schizo}. From our results in Section \ref{sec-blocks}, it is clear that $\sqrt{f_b(2k+1)}$ contains a schizophrenic pattern in its $b$-ary expansion that grows in size with $k$. This observation leads to the following Definition.
	
	\begin{definition} [Schizophrenic sequence] \label{def-seq-schizo} Let $f_b(n) = b f_b(n-1) + n$ be a recurrence with initial value $f_b(0) = 0$ for fixed positive integer $b\geq2$. Then the monotonically increasing sequence of irrational numbers $S=(\sqrt{f_b(1)}, \sqrt{f_b(3)}, \ldots, \sqrt{f_b(2k-1)}, \ldots)$, in which each term contains a schizophrenic pattern in its $b$-ary expansion that grows in size with $k$, is schizophrenic.
	\end{definition}
	
	We illustrate the above definition with the case $b=5$. The sequence $S=(\sqrt{f_5(7)}, \sqrt{f_5(9)}, \ldots, \sqrt{f_5(23)})$ begins as follows:
	
	\begin{table}[h!] \caption{Growing schizophrenic patterns in the base-$5$ expansion of the numbers $\sqrt{f_5(n)}$} \label{tab-base5}\centering
		\begin{tabular}{|c|l|}
			\hline 
			$n$ & $\sqrt{f_5(n)}$ \\ 
			\hline 
			$7$ & $1111.1102030301340212321423323443031320022421310240_5$ \\ 
			\hline 
			$9$ & $11111.111010303030100244100302243334320304302441412_5$ \\ 
			\hline 
			$11$ &$111111.11110003030303000302132433034013044313334032_5$ \\ 
			\hline 
			$13$ &$1.1111111111044030303030244012441021320101332242102_5\times5^6$ \\ 
			\hline 
			$15$ &$1.1111111111110430303030303024241021324410201331002_5\times5^7$  \\ 
			\hline 
			$17$ &$1.1111111111111104203030303030302412124410213244033_5\times5^8$ \\ 
			\hline 
			$19$ &$1.1111111111111111041030303030303030234430213244102_5\times5^9$ \\ 
			\hline 
			$21$ &$1.1111111111111111110400303030303030303023310244102_5\times5^{10}$ \\ 
			\hline 
			$23$ &$1.1111111111111111111103403030303030303030302311402_5\times5^{11}$ \\ 
			\hline 
		\end{tabular}
	\end{table}
	
	As we can see, the schizophrenic pattern in grows in size with $n$ - the repeating digit blocks increase in length and new blocks are formed containing new repeating digit sequences. The sequence  $S$ of irrational numbers is thus schizophrenic. 
	
	In the following final section we will present computer code that we used to obtain our results in previous sections.

	\section{Mathematica code}
	
	As stated earlier, all numerical computations presented in this paper were performed with Wolfram Mathematica $11.1$. In this section we present some of the code we used to obtain schizophrenic numbers in various bases. Note that the output generated by Mathematica uses the same notation as this paper for digits greater than $9$, i.e., $10$ is denoted by $a$, $11$ by $b$, and so forth. 
	
	We begin by defining $f_b(n)$ as the function \texttt{f} below using the equation of its solution, $f_b(n) = \frac{b^{n+1}-b(n+1)+n}{(b-1)^2}$. We then define another function \texttt{g} in order to display $\sqrt{f_b(n)}$ in base $b$ with arbitrary precision. Thus,
	
	\begin{verbatim}
	
	Clear[f]; Clear[g];
	f[b_, n_] := (b^(n + 1) - b*(n + 1) + n)/(b - 1)^2;
	g[b_, n_, precision_] := 
	BaseForm[NumberForm[N[Sqrt[f[b, n]], precision], precision], b];
	
	\end{verbatim}
	
	For instance, getting the value of $\sqrt{f_{13}(49)}$ with a precision of $140$ digits, i.e.,
	
	\begin{verbatim}
	
	g[13, 49, 140]
	
	\end{verbatim}
	
	gives
	\begin{center}
		\begin{tabular}{rllll}
			$1.1111111111$& $1111111111$& $1111111111$& $1111111111$& $1111110c20$ \\
			$7070707070$  & $7070707070$& $7070707070$& $7070707070$& $70706baa16$ \\ 
			$ba205386ba$  & $205386ba20$& $5386ba2053$& $86ba205386c$& $\ldots_{13}$ \\
			$\times13^{24}$.
		\end{tabular}
	\end{center}
	
	In Definition \ref{def-seq-schizo} we introduced schizophrenic sequences. Our example was the base-$5$ schizophrenic sequence $S=(\sqrt{f_5(7)}, \sqrt{f_5(9)}, \ldots, \sqrt{f_5(23)})$, presented in Table \ref{tab-base5}. To obtain the values therein, we used the following one-line code:
	
	\begin{verbatim}
	
	Table[g[5, n, 50], {n, 7, 23, 2}]
	
	\end{verbatim}

\end{document}